\documentclass[11pt,reqno,twoside]{amsart}
\usepackage{amsmath,amsthm, amscd, amssymb, amsfonts, palatino, mathrsfs}
\usepackage{cite}
\usepackage[utf8]{inputenc}
\usepackage{color}
\usepackage[all]{xy}
\usepackage{tikz}
\usetikzlibrary{matrix,arrows,decorations.pathmorphing}
\usepackage{tikz-cd}
\usepackage{hhline}
\usepackage{rotating}
\usepackage{textcomp}
\usepackage{rotating}
\usepackage{relsize}
\usepackage{multirow}
\usepackage{caption}
\usepackage{float}

\usetikzlibrary{calc, positioning}

\newcommand{\Aut}{\operatorname{Aut}}

\numberwithin{equation}{section}\theoremstyle{plain}

\newtheorem{theorem}{Theorem}[section]
\newtheorem{lemma}[theorem]{Lemma}

\newtheorem{proposition}[theorem]{Proposition}

\theoremstyle{definition}
\newtheorem{definition}[theorem]{Definition}

\theoremstyle{remark}
\newtheorem{obs}[theorem]{Remark}
\newtheorem{remark}[theorem]{Remark}

\newcommand\id{\operatorname{id}}

\def\pf{\begin{proof}}
\def\epf{\end{proof}}

\theoremstyle{remark}

\title [Exact factorization of finite Groups]{On the number of exact factorization of finite Groups}

\author[Ochoa Arango]{Jes\'us Alonso Ochoa Arango}
\email{jesus.ochoa@javeriana.edu.co}
\author[Umbarila Mart\'in]{Mar\'ia Ang\'elica Umbarila Mart\'in}
\email{m\_umbarila@javeriana.edu.co}

\address{\noindent
Mathematics Department, Faculty of Sciences,
Pontificia Universidad Javeriana. Bogot\'a, Colombia.
}

\date{\today}

\begin{document}

\renewcommand{\baselinestretch}{1.2}
\thispagestyle{empty}

\begin{abstract}
In this work, we study the function $f_2(G)$ that counts the number of exact factorizations of a finite group $G$. We compute $f_2(G)$ for some well-known families of finite groups and use the results of Wiegold and Williamson \cite{WW} to derive an asymptotic expression for the number of exact factorizations of the alternating group $A_{2^n}$. Finally, we propose several questions about the function $f_2(G)$ that may be of interest for further research.
\end{abstract}
 
\maketitle

\section*{Introduction}
The determination of all factorizations of a given group is a longstanding question in group theory. In fact, a well-known theorem by Burnside regarding the solvability of finite groups \cite[Thm 3.10]{Isa} is perhaps the first significant result concerning factorizations of finite groups. Following this, several important results emerged throughout the 20th century, particularly those by Ito and Kegel \cite{Ito, Kegel}, which demonstrated that various structural properties of a group could be inferred from information about its factorizations. Specifically, Ito showed that if $G=AB$ with $A$ and $B$ abelian groups then $G$ is metabelian; Kegel, on his part, proved that if $G=AB$ where $A$ and $B$ are nilpotent groups then $G$ is solvable. Additionally, Ore previously established that if $A$ and $B$ are a pair of maximal conjugate subgroups of $G$ then $G=AB$ and, conversely, every maximal factorization of $G$ can be obtained in this way.

Recent works related to factorizations of finite groups have focused on the case where $G$ is a simple group. For example, in \cite{LPS}, Liebeck \emph{et. al} identify all factorizations $G=AB$ where $G$ is a finite group such that there exists a subgroup chain $L \triangleleft G \leq  \text{Aut}( L)$, with $L$ with being a finite simple group. Furthermore, in these factorizations, $A$ and $B$ are maximal subgroups that do not contain $L$.

The study of group factorizations is not only of interest within group theory itself; it also deeply influences the study of other algebraic structures. For example, the construction of Hopf algebra extensions of the algebra of functions over a group by a group algebra is equivalent to finding matched pairs of finite groups \cite{Ma2, Ma}, a concept closely related to group factorization. Non-exact factorizations of groups have proven relevant in the classification of a broad class of finite double groupoids, known as slim double groupoids \cite{AN}. This class of double groupoids generalizes the vacant double groupoids introduced by K. Mackenzie \cite{Mack92}, whose category is equivalent to the category of matched pairs of groupoids.
In the specific case where the total base of the double groupoid is a single point, the problem of classifying vacant double groupoids reduces to the classification of matched pairs of groups, which in turn is equivalent to the classification of exact factorizations of finite groups. Similarly, determining all finite slim double groupoids with a single-point base is equivalent to determining all factorizations (both exact and non-exact) of finite groups. These results have also been extended to the smooth setting, allowing for the construction of a wide family of examples of double Lie groupoids \cite{AOT}.

In recent years, several works in the area of probabilistic group theory \cite{Gus, Les1, Les2, Tu}, among others, have shown that the \emph{degree of commutativity of a group} is intrinsically related to the number of factorizations of a finite group, both exact and non-exact. The degree of commutativity is an invariant associated with finite groups that measures the probability that two randomly chosen subgroups of a group $G$ commute, or equivalently, the probability that their product is again a subgroup of $G$.

On the other hand, F. Saeedi and M. Farrokhi D.G., in \cite{SaFa}, have studied factorizations (not necessarily exact) of certain families of groups such as dihedral groups $D_{2n}$, generalized quaternions $Q_{2n}$, and the modular groups, obtaining a series of explicit formulas for the total number of factorizations. However, these formulas, in principle, do not have any interpretation in terms of other classical arithmetic functions in number theory. Therefore, a natural problem arises: counting the exact number of factorizations for different families of groups, which is the problem we will address in this article. In what follows, we will describe the structure of the paper.

In the first section, we review key concepts and results related to the notions of  \emph{exact factorization} and \emph{matched pair of groups}, providing several relevant examples to illustrate these concepts. We also begin the study of the function \emph{number of exact factorizations} of a finite group $G$, denoted by $f_2(G)$, and compute the values of $f_2(G)$ for two class of groups: cyclic and dihedral. 

The second section is dedicated to the calculation of $f_2(G)$ for several well-known families of $p$-groups: quaternions, generalized quaternions, semi-dihedral groups, and modular $p$-groups. The third section is entirely devoted to computing $f_2(\text{PSL}_2(\mathbb{F}_q))$, closely following the work of Ito \cite{Ito}. We will show in Theorem \eqref{finite-PSL} that, in almost all cases, it equals $1$.

In the fourth section, we review some results by Wiegold and Williamson \cite{WW} regarding the exact factorizations of the alternating group $A_{k}$ to determine the exact value of $f_2(A_{2m})$ in Proposition \ref{prop-A2m}. Then, in Theorem \ref{teo-A2n}, we study the asymptotic behaviour of of $f_2(A_{2^n})$.

We conclude the paper with some results of a number-theoretical nature related to exact factorizations. We also present a table that includes all the exact factorizations of groups of order less than or equal to $20$ and close the section with a conjecture about the sum of the values of $f_2(G)$ or all groups of a given order $n$.

\begin{section}{The number of exact factorizations of group $G$.}\label{number of exact factorization}

\begin{definition}
If $G$ is a finite group, a \emph{factorization of $G$} is a pair of subgroups $(H,K)$ such that $G = HK$. If, moreover, $H \cap K = 1$ then the factorization it is said to be \emph{exact}.
\end{definition}

If $(H,K)$ is a factorization of a finite group then it is clear that $(K,H)$ is another one. Since we are interested in the counting on how many pairs appear as factorizations of $G$, the two above factorizations should be counted as only one. Hence, we introduce the following definition

\begin{definition}
If $G$ is a finite group and $(H, K)$ and $(L, M)$ are two exact factorizations of $G$, they are \emph{equivalent factorizations} if there are isomorphisms of groups $H \cong L$ and $K \cong M$, or $H \cong M$ and $K \cong L$; in other case, they are called \emph{inequivalent}.
\end{definition}

It is clear that above definition introduce an equivalence relation on the set of all exact factorization of a group $G$. The collection of all this classes will be denoted by $\mathcal{EF}(G)$.

\begin{definition}
Given a finite group $G$ we define the \emph{rough number of exact factorizations of $G$}, which we denote by $f_2(G)$, as the cardinal of the set $\mathcal{EF}(G)$; that is, the total number of inequivalent exact factorizations of $G$.
\end{definition}

As it is common we will denote by $\omega(n)$ the total number of prime divisors of $n$. A well known fact in group theory is that $\mathbb{Z}_m\oplus\mathbb{Z}_n \simeq\mathbb{Z}_{mn}$ if $\gcd(m,n)=1$. This results can be paraphrased in terms of factorizations of groups as follows.

\begin{lemma}\label{factzn}
 If $n$ is a positive integer with prime factorization $n=p_1^{\alpha_1}p_2^{\alpha_2}...p_k^{\alpha_k}$, then the exact factorizations of the cyclic group $\mathbb{Z}_n$ are of the form $\mathbb{Z}_I\mathbb{Z}_J$, where $I$ is the product of   elements in a subset $\hat{I}\subseteq \{p_1^{\alpha_1},...,p_k^{\alpha_k}\}$ and $J$ the product of the elements in the complement of $\hat{I}$. 
\end{lemma}
\begin{proof}
Let us consider $r$ a generator of $\mathbb{Z}_n= \langle r \rangle$ and, as was introduced in the statement of the proposition, given a subset $\hat{I}\subseteq \{p_1^{\alpha_1},...,p_k^{\alpha_k}\}$ denote by $I$ the product of its elements and by $J$ the product of the elements in the complement of $\hat{I}$.  It is clear that $\langle r^I \rangle$ and $\langle r^J \rangle$ are isomorphic copies of $\mathbb{Z}_J$ and $\mathbb{Z}_I$ inside $\mathbb{Z}_n$. If there exists $r^k \in \mathbb{Z}_I\cap \mathbb{Z}_J$ then there are integers $a$ and $b$ such that $n|k-bJ$  and $n|k-aI$ but $n=IJ$, then we can ensure that $I|k$ and $J|k$ and, as a consequence, $k$ is divisible by the least common multiple of $I$ and $J$. This, together with the fact $(I,J)=1$, implies that $n|k$ and therefore $r^k=1$. Hence $\mathbb{Z}_I\cap\mathbb{Z}_J = \{1\}$ and by a cardinality argument we can conclude that $\mathbb{Z}_n=\mathbb{Z}_I \mathbb{Z}_J$.

Conversely, if $\mathbb{Z}_n= H L$ is an exact factorization of $\mathbb{Z}_n$ then $H$ and $L$ are cyclic subgroups, isomorphic to $\mathbb{Z}_h$ and $\mathbb{Z}_\ell$ respectively. Clearly $\gcd(h,\ell)=1$ because in other way, the structure theorem of the cyclic groups implies that they should have non trivial intersection. Hence $h$ and $\ell$ are products of complimentary subsets of $\{p_1^{\alpha_1},...,p_k^{\alpha_k}\}$.
\end{proof}
\begin{theorem}\label{teoz}
 The number of exact factorizations of the the cyclic group $\mathbb{Z}_n$ is given by $f_2(\mathbb{Z}_n)=2^{\omega(n)-1}-1$.
\end{theorem}
\begin{proof}
Lemma \ref{factzn} provide us with all possible factorizations of the group $\mathbb{Z}_n$ and the problem to count such factorizations boils down to the counting of all possible choices of proper subsets $I$ of $N=\{p_1^{\alpha_1},...,p_k^{\alpha_k}\}$. This number is 
$$ \binom{\omega(n)}{1}+\binom{\omega(n)}{2}+...+\binom{\omega(n)}{\omega(n)-1},$$ 
and because of the symmetric role played by $I$ and its complement $J=N-I$ in an exact factorization of groups, that is, since  $\mathbb{Z}_I\mathbb{Z}_J = \mathbb{Z}_J\mathbb{Z}_I$, then we get 
$$f_2(\mathbb{Z}_n)=\frac{1}{2}\left \{ \binom{\omega(n)}{1}+\binom{\omega(n)}{2}+...+\binom{\omega(n)}{\omega(n)-1} \right \};$$
that is $f_2(\mathbb{Z}_n)=2^{\omega(n)-1}-1$.
\end{proof}

The above result can be extended to a more general abelian groups by using the \emph{primary decomposition theorem for finite abelian groups}. Let $G$ be a finite abelian group of order $n >1$ and let the unique factorization into primes be $n=p_1^{\alpha_1} p_2^{\alpha_2} \cdots p_k^{\alpha_k}$.	Then
\begin{enumerate}
\item $G \cong A_1 \times A_2 \times \cdots \times A_k$ where $\mid A_i\mid = p_i^{\alpha_i}$, for every $i = 1 , \ldots, k$.
\item For each $A \in \left\{ A_1, A_2 , \ldots , A_k \right\}$ con $\mid A \mid = p^\alpha$ we have
\begin{equation*}
A \cong \mathbb{Z}_{p^{\beta_1}} \oplus \mathbb{Z}_{p^{\beta_2}} \oplus \cdots \oplus \mathbb{Z}_{p^{\beta_t}},
\end{equation*}
with  $\beta_1 \geq \beta_2 \geq \cdots \geq \beta_t \geq 1$ and $\beta_1 + \beta_2 + \cdots + \beta_t = \alpha$. Here $t :\equiv t(p)$ and $\beta_j :\equiv \beta_j(p)$.
\item The decompositions in the above items are unique up to isomorphism. In fact, the $\beta's$ in the second item are the invariant factors of each $A$.
\end{enumerate}

\begin{theorem}
If $G$ is a finite abelian group then, with the notation introduced above, the number of distinct exact factorizations of $G$ is 
\begin{equation}
f_2(G)=(2^{\omega(n)-1}-1)\prod_{p\mid n}\left( 2^{t(p)-1}-1 \right).
\end{equation}
\end{theorem}
\begin{proof}
By the \emph{primary decomposition theorem for finite abelian groups} we only have to count the many ways in which we can pull apart the factors of the decomposition $G \cong A_1 \times A_2 \times \cdots \times A_k$ into two separate groups, by choosing one primary component $A$ attached to a prime $p$ dividing $n$. Once we have done this, we have to count again the ways in which we can pull apart de factors of the primary component chosen $A \cong \mathbb{Z}_{p^{\beta_1}} \oplus \mathbb{Z}_{p^{\beta_2}} \oplus \cdots \oplus \mathbb{Z}_{p^{\beta_t}}$ .

The counting pointed in the above paragraph can be done in a similar way to the one made in the proof of lemma \eqref{factzn}, and in this way we can get the expression 
\begin{equation*}
f_2(G)=(2^{\omega(n)-1}-1)\prod_{p\mid n}\left( 2^{t(p)-1}-1 \right).
\end{equation*}
This argument is also supported in the fact that a cyclic group of prime power order $\mathbb{Z}_{p^{\alpha}}$ cannot be factored (as an exact product of two subgroups ) because their subgroups are all totally ordered by inclusion. 
\end{proof}

The above results exhaust the computation of the funcion $f_2(G)$ for the case of all finite abelian groups. Now we are going to move in the next step to some well known families of groups that are simple enough to compute their number of exact factorizations.

\begin{remark}
If $D_{2n}$ denotes the dihedral group of order $2n$ then it can be shown that all the subgroups of $D_{2n}$ are cyclic or other dihedral. For a detailed proof on these fact the reader can consult \cite{Con}.
\end{remark}

\begin{lemma}\label{FACD2N}
Let $n$ be a positive integer with prime factorization $n=p_1^{\alpha_1}p_2^{\alpha_2}...p_k^{\alpha_k}$. Let us denote by $\mathscr{P}$ the set $\{p_1^{\alpha_1},...,p_k^{\alpha_k}\}$ or the set $\{2^{\alpha_1-1},p_2^{\alpha_2},...,p_k^{\alpha_k}\}$ depending if $n$ is odd or even, with $p_1=2$. The only exact factorizations of the dihedral group $D_{2n}$ are: 
 \begin{enumerate}
     \item If $n$ is odd, then $D_{2n}=D_{2I}\mathbb{Z}_J$, where $I \subseteq \mathscr{P}$ and $J=\mathscr{P}-I$
     \item If $n$ is even, then we have the factorization $D_{2n}=D_{2I}\mathbb{Z}_J$ and also the factorization $D_{2n}=D_{2I}D_{2J}$ where $I\subseteq \mathscr{P}$ and $J=\mathscr{P}-I$.
 \end{enumerate}
 \end{lemma}
 \begin{proof}
 Let $D_{2n}$ be the dihedral group of order $2n$ and, as usual, write down the presentation
    \begin{center}
        $D_{2n}=\langle r,s | r^n=s^2=e$  and  $rs=sr^{-1} \rangle.$
    \end{center}
Remind that $I$ stands for the subset of $\mathscr{P}$ or for the product of its elements and which one we are using will be clear from the context. The order of $r^J$ is $|r^J|=I$ and  for $k<n$,
    \begin{align*}
    (r^J)^ks&=sr^{(n-Jk)} \\
    &=sr^{J(I-k)} \\
    &=s(r^J)^{(I-k)};
\end{align*}
hence  $\langle r^J,s \rangle \simeq D_{2I}$ and $\langle r^I \rangle \simeq \mathbb{Z}_J$.
   
If we suppose  $x$ is an element in the intersection $\langle r^J, s \rangle \cap \langle r^I \rangle $ then  $x$ is of the forms $(r^J)^c$ and $(r^I)^d$ for some positive integers $c<I$ and $d<J$. Therefore $n\mid Jc-Id$ and, since $\mid Jc-Id \mid <n$, it follows $J\;c=I\;d$. The fact $(J,I)=1$ implies $J$ divides $d$ and, as a consequence, $a=r^{I\;d} = e$. It follows that $\langle r^J, s \rangle \cap \langle r^I \rangle $ is trivial and by cardinality $D_{2n}=D_{2I}\mathbb{Z}_J$.

Let us suppose $n$ is an even integer, then $|r^{2J}|=I$ and
 \begin{align*}
     r^{2J}(sr^{2I})&=sr^{-2J} r^{2I} \\
     &=(sr^{2I})r^{-2J}.
 \end{align*}
Thus the subgroup $\langle r^{2J}, sr^{I} \rangle$ is the dihedral group $D_{2I}$ and in the same way we prove $\langle r^{2I},s \rangle$ is the dihedral group $D_{2J}$. If $y$ is an element in the intersection $\langle r^{2J}, sr^{I} \rangle \cap  \langle r^{2I},s \rangle $ then $y=(r^{2J})^l(sr^{I})$ and  $y=(r^{2I})^ks$,  with $0 \leq l < I$ and $ 0 \leq k < J$. Equating this two expressions
 \begin{align*}
     (r^{2J})^lsr^{I}&=(r^{2I})^ks \\
     r^{2Jl-I}s&=r^{2Ik}s \\
     r^{2Jl-I}&=r^{2Ik},
\end{align*}
hence  $2JI \mid   \left(2Jl-I(2k+1)\right)$. However $2Jl < 2JI$ and $(2k+1)I\leq 2JI$, thus
 \begin{align*}
     -2JI < 2Jl-(2k+1)I & < 2JI \\
     \mid 2Jl-(2k+1)I\mid & < 2JI,
 \end{align*}
which is a contradiction unless $2Jl =(2k+1)I$. In this case, since $(I,J)=1$, it follows that $2J$ divides $2k+1$, which is impossible and this finish the proof.
 \end{proof}
 
 \begin{theorem}\label{teodn}
The number of exact factorizations of the dihedral groups are given by the expressions 
  \begin{enumerate}
      \item $f_2(D_{2n})=2^{\omega(n)}-1$, if $n$ is odd; 
      \item $f_2(D_{2n})=2^{\omega(n)}+2^{\omega(\frac{n}{2})-1}+\omega(\frac{n}{2})-\omega(n)-1$, if $n$ is even.
  \end{enumerate}
 \end{theorem}
 \begin{proof}
As in lemma \eqref{FACD2N}, let $n=p_1^{\alpha_1}p_2^{\alpha_2}...p_k^{\alpha_k}$ be the factorization of $n$ into prime factors. Let us write $\mathscr{P}$ for the set $\{p_1^{\alpha_1},...,p_k^{\alpha_k}\}$ or the set $\{2^{m-1},p_2^{\alpha_2},...,p_k^{\alpha_k}\}$ depending if $n$ is odd or even, with $p_1=2$ and $\alpha_1=m$. 

We will divide the proof in two parts. If $n$ odd then, according to Lemma \ref{FACD2N}, all possible factorizations of $D_{2n}$ are of the form $D_{2n}=D_{2I}\mathbb{Z}_J$, where $I \subseteq \mathscr{P}$ and $J=\mathscr{P}-I$.  The counting of these factorizations boils down to the counting of all possible choices of $I$ as a proper subset of $\mathscr{P}$. This number is 
$$\binom{\omega(n)}{0}+\binom{\omega(n)}{2}+...+\binom{\omega(n)}{\omega(n)-1}.$$ 
That is $f_2(D_{2n})=2^{\omega(n)}-1$.

If $n$ is even, we are going to divide the analysis depending on if $4$ divides or doesn't divides $n$. The above argument implies that, at least, $f_2(D_{2n}) \geq 2^{\omega(n)}-1$. If $n=2^mp_2^{\alpha_2}...p_k^{\alpha_k}$ then lemma  \ref{FACD2N} said that, in this case, if $4 \mid n$ there is a new type of factorizations $D_{2n}=D_{2I}D_{2J}$ and the counting of these  factorizations depend of choices of $I$  and $J$ as a proper subsets of $\mathscr{P}$.  If we remember that $D_{2I}D_{2J}=D_{2J}D_{2I}$ then the possible ways in which we can do this choices is  $2^{\omega(\frac{n}{2})-1}$ and the total number of factorizations is $f_2(D_{2n})= 2^{\omega(n)}-1+2^{\omega(\frac{n}{2})-1}$. Finally, if $4$ doesn't divides $n$ and if we take $I=\{p_2^{\alpha_2},...,p_k^{\alpha_k}\}$ then this choice gives rise to a factorization $D_{2n}=D_{2I}\mathbb{Z}_2$, which was included in the first type of factorizations. It follows that in this case the total number of factorizations is $f_2(D_{2n})= 2^{\omega(n)}-1+2^{\omega(\frac{n}{2})-1}-1$. The results just obtained can be summarized in the expression $f_2(D_{2n})=2^{\omega(n)}+2^{\omega(\frac{n}{2})-1}+\omega(\frac{n}{2})-\omega(n)-1$.
 \end{proof}

In the following table the reader can find some examples of factorizations that illustrate better the  analysis carried out for the dihedral group. 
\begin{center}
\begin{tabular}{|c|c|c|c|c|}
 \hline
 \multicolumn{5}{|c|}{Factorizations of some dihedral groups} \\
 \hline
 $n$ & $D_{2n}$ & Subgroups  & Factorizations  & Generators\\
 \hline
$9$ & $D_{18}$ &
    $\mathbb{Z}_2$,
    $\mathbb{Z}_3$,
    $D_6$ y
    $\mathbb{Z}_9$
  & $D_{18}=\mathbb{Z}_9\mathbb{Z}_2$ & $\mathbb{Z}_9=\langle r \rangle$ y $\mathbb{Z}_2=\langle s \rangle$ \\ 
\hline
$10$ & $D_{20}$ & $\begin{matrix}
\mathbb{Z}_2, \mathbb{Z}_{10}, \mathbb{Z}_2\times\mathbb{Z}_2, \\
\mathbb{Z}_5 \;\text{y}\; D_{10}
\end{matrix}$
 &  $\begin{matrix}
    D_{20}=\mathbb{Z}_{10}\mathbb{Z}_2 \\
    D_{20}=D_{10}\mathbb{Z}_2 \\
    D_{20}=(\mathbb{Z}_2\times\mathbb{Z}_2)\mathbb{Z}_5 \\
\end{matrix}$ &
$\begin{matrix}
    \mathbb{Z}_{10}=\langle r \rangle \;\text{y}\; \mathbb{Z}_2=\langle s \rangle \\
    D_{10}=\langle r^2, s \rangle \;\text{y}\; \mathbb{Z}_2=\langle r^5 \rangle \\
    \mathbb{Z}_2\times\mathbb{Z}_2= \langle r^5, s \rangle  \;\text{y}\; \mathbb{Z}_5= \langle r^2 \rangle
\end{matrix}$\\
\hline

\end{tabular}
\end{center}

\vspace{0.5cm}

To finish  this section we will study other well known family of groups the \textit{generalized quaternions}, that fits into the class of dycyclic groups.

\section{Exact Factorizations of some finite p-groups}

If $p$ is a prime number, non abelian finite $p$-groups $G$, with a maximal cyclic subgroup, can be grouped into some well known families \cite[Thm 4.1]{Su}: the modular $p$-group $M(p^n)$, if $p$ is odd and
if $g=2^n$, then we have the dihedral group $D_g$, the generalized quaternion group $Q_g$, the modular group $M(g)$ or the generalized dihedral $S_g$.

\begin{definition}\label{generalized-quat}
The group of generalized quaternions, denoted by $Q_{2^n}$, is the group of order $g=2^n$ with the following presentation  $$\mathcal{Q}_{2^n}=\langle a,b \quad |\quad a^{2^{n-1}}=1,\quad  bab^{-1}=a^{-1} \quad \text{and} \quad b^2=a^{2^{n-2}} \rangle,$$
and that $b^{2}$ is the only element of $Q_{2^n}$ of order two.
\end{definition}

\begin{proposition}\label{teoq}
The group $Q_{2^n}$ of generalized quaternions doesn't have any exact factorizations.
 \end{proposition} 
 \begin{proof}
 If $Q_{2^n}=GH$ is an exact factorization, then  $|H|=2^k$ and $|G|=2^q$ with $k,q < n$. Hence both subgroups must contain an element with order $2$ but, since the only element with that order in $Q_{2^n}$ is $b^2$, it must be a common element of $G$ and $H$ and then the factorizations cannot be exact. In consequence $Q_{2^n}$ doesn't have exact factorizations. 
 \end{proof}

In the rest of this section we will study other well known family of finite groups, the \emph{semidihedral} ones. These are one of the $2$-groups with the property of having a maximal cyclic subgroup. For a more detailed treatment of this family and other ones related to this the reader can refer to \cite{Hup} or \cite{Su}.

\begin{definition}
If $g = 2^n$ with $n \geq 3$, and $a=2^{n-2}$ then we define the \emph{Quasidihedral} or \emph{Semidihedral} group $SD_g$ as the finite group whose presentation in terms of generators and relations is the following
\begin{equation}
SD_g:=\langle x, y \mid x^{2a}=y^2 = 1 \quad ; \quad y^{-1}xy=x^{a-1} \rangle .
\end{equation}
\end{definition}
\begin{remark}\label{properties-quasi-dihedral}
In \cite{Hup} or \cite{Su} it is shown that $SD_g$ has the following properties:
\begin{enumerate}
\item It has a maximal cyclic subgroup;
\item The group  $SD_g/\Phi(SD_g)$ is an abelian group of type $(2,2)$, where $\Phi(G)$ denotes the Frattini subgroup of $G$, i.e, the intersection of all the maximal subgroups of $G$;
\item The subgroups $\Phi(SD_g)$, the derived group $[SD_g,SD_g]$ and $Z(SD_g)$ are cyclic;
\item $Z(SD_g)$ has order $2$ and $SD_g/Z(SD_g) \simeq D_{2^{n-1}}$;
\item $[SD_g,SD_g] = \Phi(SD_g)$;
\item The only non cyclic maximal subgroups of $SD_g$, up to isomorphism, are $D_{2^{n-1}}$ and $Q_{2 ^{n-1}}$.
\end{enumerate}
\end{remark}
Moreover, in the same references it is shown that every subgroup $H$ of $SD_g$ satisfies one of the following conditions:
\begin{enumerate}
\item $Z(SD_g) \subseteq H$ or
\item $H$ is one of the cyclic groups 
$$\{ e \}, \langle y \rangle, \langle x^2y \rangle, \cdots , \langle x^{2^{m-1}-2}y \rangle.$$
\end{enumerate}
 \begin{theorem}\label{fact-quasi-dihedral}
If $g = 2^n$ with $n \geq 4$, and $a=2^{n-2}$ then the number of exact factorizations of the quasdihedral groups $SD_g$ are given by
$$f_2(SD_g)= 2.$$
 \end{theorem}
 \begin{proof}
If the quasidihedral group $SD_g$ can be writen as $SD_g=HK$, with $H \cap K = e$ then:

\textit{Case I:} If $Z(SD_g) \leq K$ then, as we are looking for exact factorizations, by the remark above \eqref{properties-quasi-dihedral}, $H$ can be only a subgroup of the form $\langle x^i y \rangle$ with $i=1, \ldots,2^{m-1}-2$.

Since $|x^{2i}y| = 2$ and $|x^{2i+1}y|=4$ then $H$ is of order $2$ or $4$. In the first case, if $H=\langle x^{2i} y \rangle$ then the order of $K$ is $2^{n-1}$ and, as a consequence, $K$ is a maximal subgroup. Hence, the only possible options for $K$ is to be one of the following subgroups:
\begin{itemize}
\item $K\simeq \mathcal{C}_{2^{n-1}}$, the cyclic group of order $2^{n-1}$; 
\item $K \simeq D_{2^{n-1}}$, a dihedral group;
\item $K \simeq \mathcal{Q}_{2^{n-1}}$, a generalized quaternion group. 
\end{itemize}
For the first option we can take $H=\langle x \rangle $ and $K= \langle y \rangle $. In the second one we have $K = \langle x^{2\ell+1}, x^{2m}y\rangle$ and since the odd powers of $x$ also generate the same subgroup of $x$ then $y \in K$; thus $K = \langle x, y\rangle$ which is impossible because of the triviality of the intersection of $H$ and $K$. In the last option, if $K \simeq \mathcal{Q}_{2^{n-1}}$ we have that $K$ should be generated by elements of order $2^{n-2}$ and $4$ satisfying the defining relations of the generalized quaternions. This boils down to $K = \langle x^2, xy \rangle$ and with the adequate choice of $H=\langle x^2y \rangle$ we get another exact factorization of $SD_g$.

If $H = \langle x^{2i+1}y \rangle$ then $|K| = 2^{n-2}$ and going to the quotient $D_{2^{n-1}} \simeq \dfrac{SD_g}{Z(SD_g)} = \dfrac{K}{Z(SD_g)} \overline{H}$, is an exact factorization of the dihedral group $D_{2^{n-1}}$. Lema \eqref{FACD2N} guarantees that this is impossible unless $K=Z(SD_g)$ and $n=3$, but this a contradiction.

\textit{Case II.} If no one of the subgroups $K$ or $H$ contains the center then both of then are in the list  
$$\{ e \}, \langle y \rangle, \langle x^2y \rangle, \cdots , \langle x^{2^{m-1}-2}y \rangle;$$
this implies that $|SD_g| = 8$ or $16$. But, since $n \geq 4$ then the  only possibility is $16$. If this the case, the group $SD_g$ could be factored as $SD_g=HK$ with $H$ and $K$ groups of order $4$ in the above list. This implies $H= \langle x^{2i+1}y \rangle$ and $K= \langle x^{2j+1} y\rangle$, whose intersection is $\{ e, x^{2^{n-2}}\}$. Hence the factorization cannot be exact.

From all above the only factorizations of $SD_g$ are $\mathcal{C}_{2^{n-1}} \mathcal{C}_2$ and $D_{2^{n-1}}\mathcal{C}_2$ and we get the desired result.
\end{proof}

\begin{definition}\label{defi-modular}
The \emph{modular $p$-group}, denoted by $M(p^n)$, is the group of order $p^n$ whose presentation is give by
\begin{equation}
M(p^n):=\langle x,y \mid x^{p^{n-1}}=y^p=1, \quad y^{-1}xy=x^{1+p^{n-2}} \rangle .
\end{equation}
\end{definition}
\begin{remark}
The word \emph{modular} in the above definition refers to the fact that the  lattice subgroup of $M(p^n)$ is modular. 
\end{remark}

\begin{remark}\label{remark-modular}
With the notation of definition \eqref{defi-modular}, if $Z$ denotes the subgroup $\langle x^{p ^{n-2}}\rangle$ and $H$ is a subgroup of $M(p^n)$ that doesn't includes the subgroup $Z=\langle x^{p ^{n-2}}\rangle$ then $|H|=p$ and $H= \langle x^{ip^{n-2}} y \rangle$ for some $i = 0, \ldots, p-1$.
\end{remark}
\begin{theorem}
If $p > 2$ is a prime number and $n \geq 3$ is a positive integer then the number $f_2(M(p^n))$, of exact factorizations of $M(p^n)$, is equals to $1$.
\end{theorem}
\begin{proof}
Let us write $M(p^n)=HK$ where $H$ and $K$ are subgroups of $M(p^n)$. Following the ideas of \cite{SaFa}, let us denote by $Z$ the subgroup $\langle x^{p ^{n-2}}\rangle$ and consider two cases: 

\textit{Case I: If $Z$ is not a subgroup of $H$ nor of $K$}, then by remark \eqref{remark-modular} both subgroups has order $p$ which is impossible since $n\geq 3$.

\textit{Case II: One of the subgroups $H$ or $K$ contains $Z$}, let us said $H$. Since the factorization is exact then $Z$ is not a subgroup of $K$ and again, by lemma \eqref{remark-modular}, $|K|=p$ and $|H|=p^{n-1}$. 

It is shown in \cite[Hilfssatz 8.7]{Hup} that the modular $p$-group $M(p^n)$, with $n\geq 3$, has exactly $p+1$ subgroups of index $p$ described as follows: $p$ cyclic subgroups and just one non-cyclic; this last one is equals to $F=\langle x^p, y\rangle$ (loc. cit.) . Since $x^p$ generates the center of $M(p)$ and $y$ has order $p$ it follows $F \simeq \mathbb{Z}_{p^{n-2}} \oplus \mathbb{Z}_p$. It is also clear that any subgroup of order $p$ in $M(p^n)$ is of the form $\langle x^{ip^{n-2}} y \rangle$, for $0 \leq i < p $ (see \eqref{remark-modular}), hence $ H \cap K \neq \{e\}$ and the factorization cannot be exact.
The above reasoning implies that the only possible exact factorization for the modular $p$-group is $K \simeq \mathbb{Z}_p$ and $H\simeq \mathbb{Z}_{p^{n-1}}$. 

\end{proof}

\section{A review of the case of the projective special linear groups}

In \cite[Thm. 3.3]{SaFa} the authors, based on Ito's results \cite{Ito}, faced the problem of calculate the total number of factorizations of the projective special linear group $G=\text{PSL}_2(\mathbb{F}_q)$ over a finite field $\mathbb{F}_q$. In that paper, the authors got expressions for this number in terms of the number of subgroups of $G$ and depending on the nature of the prime number $p$, where $q=p^n$.  In this section, after a careful revision of Ito's paper, we point a difference with the result in the aforementioned result of Saaedi and Farrokhi \cite{SaFa} and compute the total number of exact factorizations of $G$.

\subsection{On Ito's result about exact factorizations of $G=\text{PSL}_2(\mathbb{F}_q)$}

Let us write $q=p^n$, with $p$ a prime number. Ito's method to determine all factorizations of $G=\text{PSL}_2(\mathbb{F}_q)$, takes not only the prime number $p$ as income, but also another prime $\ell$ characterized (and whose existence is guaranteed) by Zsigmondy's result \cite{BV, Zsig}:
\begin{equation}\label{Zsigmondy}
p^{2n}\equiv 1 \mod \ell \quad \text{and} \quad p^{m}\not\equiv 1 \mod \ell \; , \; \text{for any} \; m < 2n.
\end{equation}

\begin{theorem}\label{finite-PSL}
If $G=\text{PSL}_2(\mathbb{F}_q)$, where $q=p^n$ with $p$ a prime number, then the function that counts the number of exact factorizations assume the followig values
\begin{equation*}
f_2(G)=
\begin{cases}
0 & \text{If $G=\text{PSL}_2(\mathbb{F}_{9})$}\\
2 & \text{If $G=\text{PSL}_2(\mathbb{F}_{7})$}\\
3 & \text{If $G=\text{PSL}_2(\mathbb{F}_{11})$}\\
1 & \text{In other case}\\
\end{cases}
\end{equation*}
\end{theorem}
\begin{proof}
The proof consist of a case by case verification of all the factorizations found by Ito in his paper \cite{Ito}.

$\bullet$ \textit{Case $p > 2$, not a Fermat's prime and $\ell \geq 7$. } Remind that $|G|=\dfrac{p^n(p^n-1)(p^n+1)}{2}$. Ito prove that if $G$ has a factorization $G=HK$, and hence a maximal one, then the factorization is exact, one of the groups, let us say $K=D_{p^n+1}$, is a dihedral group of order $p^n+1$ and $H$ is the normalizer $N_G(P)$ of a Sylow $p$-subgroup $P$ of $G$ and, by the way, has order $\dfrac{p^n(p^n-1)}{2}$. 

When does there exist any such factorization? In the case $\dfrac{p^n-1}{2}$ is even,which is equivalent to say that $p^n \equiv 1 \mod 4$, then the subgroups $H$ and $K$ have, at least, an element of order two in common and, in consequence, there is not such an exact factorization. In the case $\dfrac{p^n-1}{2}$ is odd, that is $p^n \equiv 3 \mod 4$,  then there is essentially only one such factorization $G= N_G(P)D$ (which is exact).

$\bullet$ \textit{Case $p =	 2$, $n \neq 3$ and $\ell \geq 7$.} In this case $n \geq 4$. Ito proved in \cite{Ito} that in this case there are only two factorizations $G=D_{2(2^n+1)}N$ and $G=ZN$, where $D_{2(2^n+1)}$ is the dihedral group of order $D_{2(2^n+1)}$, $N$ is the normalizer of a Sylow $2$-subgroup and $Z$ is the cyclic subgroup of order $2^n+1$. The only exact is the second one (see loc.cit.).

$\bullet$ \textit{Case $\ell = 5$.} In this case, Ito showed that we can only concentrate on the scenario where $\ell$ divides $p^{2n}-1$, and no other power of $\ell$ does so. By writing $p$ in the form $5s+t$ we prove that one of $p-1$, $p^2-1$ or $p^4-1$ is divisible by $5$; hence, $n \leq 2$. Explicitly, if $p \equiv 4 \mod 5$ then $n=1$; and if $p \equiv 2 \; \text{or}\; 3 \mod 5$, then $n=2$.

On the other hand, if $G$ has a factorization $G=HK$, and hence a maximal one, Ito proved that one of the factors, say $K$, should be $A_5$, the alternating group in five letters. Hence, $H$ is a subgroup of index less than or equals to $60$.

By using the classification of subgroups of $G=\text{PSL}_2(\mathbb{F}_q)$ \cite[Thm. 6.25]{SuI}, we can show that the least index of a subgroup of the group $G$ is equal to $p^n+1$ (except in the cases $p=2, 3, 5, 7, 11$), result that was known first to Galois in the case $n=1$ and further generalized by Moore and Wiman. Following the argument of the previous paragraph, since $H$ is a subgroup of index $60$, then $p^n+1 \leq 60$, that is, $p^n \leq 59$. 

Depending of the case whether $p \equiv 4 \mod 5$ or $p \equiv 2 \; \text{or}\; 3 \mod 5$ then $p \leq 59$ or $p^2\leq 59$, respectively. From here, depending on the value of the above congruence for $p$, the groups whose factorizations we must study are $LF(2,p)$ or $LF(2,p^2)$, respectively.

For $p \equiv 4 \mod 5$ then $p=19, 29$ or $59$. If $p=59$ then $H = A_5$, the alternating group in five letters, is a maximal subgroup of $G=\text{PSL}_2(\mathbb{F}_{59})$ and it also have as a subgroup a semidirect product of the form $K=\mathbb{Z}_{59} \rtimes \mathbb{Z}_{29}$. It is clear that the orders of $H$ and $K$ are relative primes and their product is equals to the order of $G$. Hence $G = HK$ and is the only possible exact factorization of $G$.

If $p=19$ then $D_{20}$ is a subgroup of $G=\text{PSL}_2(\mathbb{F}_{19})$ and the normalizer $N$ of a $19$-Sylow subgroup has order $19\cdot 9=171$. Hence $G$ has $N D_{20}$ as exact factorization. This normalizer is equals to a semidirect product $\mathbb{Z}_{19} \rtimes \mathbb{Z}_{9}$. Using the classification of subgroups of $G=\text{PSL}_2(\mathbb{F}_{19})$ Ito's found other factorizations of $G$, namely $NA_5$, but it is not exact.

If $p=29$, $N$ of a $19$-Sylow subgroup of $G=\text{PSL}_2(\mathbb{F}_{29})$ (that has order $29\cdot 14 =406$) and $S$ is the subgroup of $N$ of order $29\cdot 7 = 203$ then clearly $G$ admits the exact factorization $G=S \cdot A_5$ and this is the only one.

For $p \equiv 2 \; \text{or}\; 3 \mod 5$, we have $p= 2,3$ or $7$. But we are assuming that only the first power of $5$ divides $p^{2n}-1$, and $7^{4}-1=2^5\cdot 3 \cdot 5^2$, hence we can avoid $p=7$ and then, only remains the cases  $G=\text{PSL}_2(\mathbb{F}_{4})$ or $\text{PSL}_2(\mathbb{F}_{9})$. This groups are isomorphic to $A_6$ and $A_5$, respectively. For the group $\text{PSL}_2(\mathbb{F}_{9})$, all the factorizations found by Ito are not exact and, hence, $f_2(\text{PSL}_2(\mathbb{F}_{9}))=0$. This result also follows from proposition \eqref{prop-A2m}, as we can verify. For the group $G=\text{PSL}_2(\mathbb{F}_{4})$ the only exact factorization is $\mathbb{Z}_5 A_4$, where the octahedral group $A_4$ appears here as the normalizer of a $2$-Sylow subgroup of $G$.

$\bullet$ \textit{Case $\ell = 3$.} As in the above case, we assume that $q$ divides $p^{2n}-1$ only to the first power and $p^{m} \not\equiv 1 \mod q$ for every $m < 2n$. It is clear that $p=3k + r$, with $r \in \{0,1,2\}$ and then $p-1$ or $p^2-1$ is divisible by $3$ thus, $n=1$. This implies that in this case the group $G$ is nothing more than $\text{PSL}_2(\mathbb{F}_{p})$ with $p$ prime and again, as in the above case,  $p \leq 59$, by the aforementioned Galois theorem. By a straightforward verification we can check the only primes satisfying all these conditions together are $p= 2, 5, 11$ and $23$. We are going to count in each case how many exact factorizations there are. 

The group $\text{PSL}_2(\mathbb{F}_{2})=\mathbb{Z}_2\mathbb{Z}_3$ and that's the only factorization. The group $\text{PSL}_2(\mathbb{F}_{5})$ can be written as the product of $\mathbb{Z}_3$ and the normalizer of a $5$-Sylow subgroup $N$ that, in this case, by using the classification of all subgroups of $\text{PSL}_2(\mathbb{F}_{p})$, can be identified with a semi-direct product $\mathbb{Z}_5 \ltimes \mathbb{Z}_4$. There is another factorization by using $D_6$, the dihedral group of order $6$, but by a cardinality argument this is not an exact factorization.

For the case $G=\text{PSL}_2(\mathbb{F}_{11})$ we point first that the normalizer $N$ of a $11$-Sylow subgroup has order $55$. By using the classification of all subgroups of $\text{PSL}_2(\mathbb{F}_{p})$, the group $G$ has four factorizations $D_{12} N, \; A_4 N,\; A_5 \mathbb{Z}_{11}$ and $A_5N$. The only non-exact is the last one (its intersection has order $5$).

Finally, for the group $G=\text{PSL}_2(\mathbb{F}_{11})$, let $N$ be the normalizer of a $23$-Sylow subgroup (whose order is $253$). Ito proved that $G$ has only two factorizations $D_{24}N$ and $S_{4}N$ and, in each case, the order of the factors are coprime. Thus, both of them are exact factorization.
 
$\bullet$ \textit{Case $q = 23$.} The group $G=\text{PSL}_2(\mathbb{F}_{8})$ has two factorizations $\mathbb{Z}_9 N$ and $D_{18}N$, where $N$ is again the normalizer of a Sylow $2$-subgroup. The first one is exact, but the last one doesn't.

$\bullet$ \textit{Case $p = 2^k-1$ a Fermat's prime.} For $p=3$ the only factorization of $G=\text{PSL}_2(\mathbb{F}_{3})$ is $\mathbb{Z}_3\mathbb{Z}_4$. The group $G=\text{PSL}_2(\mathbb{F}_{7})$ has three factorizations $S_4N, \; S_4\mathbb{Z}_7$ and $D_{8}N$, with $N$ the normalizer of a Sywlow $7$-subgroup. By analyzing the order of the subgroups we can verify that only the first factorization is non-exact. 

If $p = 2^k-1$, with $k \geq 3$, then the only factorization of $G=\text{PSL}_2(\mathbb{F}_{p})$ is $D_{2^k}N$, with $N$ the normalizer of a Sylow $p$-subgroup, whose order is $(2^k-1)(2^k-2)$. Thus, the factorization is exact.
\end{proof}

\section{Asymptotic behaviour of $f_{2}$ for the family of alternating groups $A_{2^n}$}
 
Wiegold and Williamson, in \cite[Thm. A]{WW}, found all exact factorization of alternating and symmetric groups, by means of a careful analysis of the actions of primitive permutation groups.  In what follows we are going to use their results \textit{to count} the exact factorizations of some alternating groups.

\textit{Notation:} 
To study the exact factorizations of the symmetric group and its subgroup of even permutations, the alternating group, we will follow the notation of the paper \cite{WW}. For other notions related to finite group theory the reader can refer \cite{Rot}. If $\Omega$ is a finite set of cardinality $n \geq 3$, we are going to denote by $S^{\Omega}$ and $A^{\Omega}$ the symmetric and alternating group on $\Omega$, respectively.

\begin{definition}
If $\sigma$ is a permutation in $S^\Omega$ we define $\text{supp}(\sigma)$, the support of $\sigma$, as the set of all elements $x \in \omega$ such that $\sigma(x) \neq x$. We define the degree of $\sigma$ as the cardinality of $\text{supp}(\sigma)$.
\end{definition}
%
%
%

As a consequence of the Chebyshev's theorem (or from the prime number theorem) we can state the following refinement of Bertrand's conjecture (actually a theorem since the work of Chebyshev but, for historical reasons, it has kept the denomination of conjecture) and whose proof can be found in \cite{Ro}.

\begin{theorem}\label{teon}
If $m$ is  an integer biggest than $8$ then there is three prime numbers $r, q$ and $p$ such that the inequality $m<r<q<p<2m$ holds.
\end{theorem}

With the aim to study exact factorizations of the alternating group we need some basic notions related to permutation representations of a group acting on a set $\Omega$. 

\begin{definition}
If $G$ is a finite group acting on a set $\Omega$ then a subset $\Delta \subset \Omega$ is said a \emph{block}  if the following conditions holds
\begin{center}
$g\Delta=\Delta$ \quad or \quad $g\Delta\cap \Delta=\varnothing$, \quad for every $g \in G$.
\end{center}
A block $\Delta$ it is called trivial if it consists of a single element or if it is the whole set $\Omega$.
\end{definition}

\begin{definition}
A set $\Omega$ endowed with a transitive action of a group $G$ that doesn't has non trivial blocks it is said to be $G$-\textit{primitive}. 
\end{definition}

\begin{definition}
Let $G$ be a group acting over a set $\Omega$ of $n$ elements. If $k\leq n$ is a positive integer, we said that $\Omega$ is a $k$-homogeneous $G$-set if $G$ acts transitively over the collection of all subsets of $\Omega$ of cardinality $k$. The set $\Omega$ it is said to be $k$-transitive if $G$ acts transitively over the collection of all ordered $k$-tuples of (different) elements of $\Omega$.
\end{definition}

\begin{definition}
Let $G$ be a group acting on a set $\Omega$ of $n$ elements. Given a positive integer $k\leq n$ we say that $\Omega$ is a sharply $k$-transitive $G$-set if $\Omega$ is $k$-transitive and the action of $G$ over the set of $k$-tuples of (distinct) elements of $\Omega$ is free. That is, the only element whose action has fixed point is the identity. 
\end{definition}

\begin{remark}\label{note}
If $K=S^{\Omega}$ or $A^{\Omega}$, let $K=GH$ be an exact factorization. Given an integer $n \geq 8$ let us denote by $p$ the largest prime number less than $n-2$ and, furthermore, for $n=3$, $4$ or $5$ we consider $p =3$ . Also, for $n=6$ or $7$ let's take $p=5$.

Since $p\mid |K|$ we can assume, without loss of generality, that $p \mid |H|$. Hence, $H$ contains a $p-$cycle and we can consider the set $\Gamma \subset \Omega$ defined as the $H$ orbit containing the support of such a $p-cycle$. From now on $H^{\Gamma}$ will denote the group $H$ considered as a permutation group of $\Gamma$; that is, every $h \in H$ is a permutation of $h: \Omega \to \Omega$ then, since $\Gamma \subseteq \Omega$ is a $H$-stable subset of $\Omega$, restricting $h$ to $\Gamma$ gives rise to a permutation $h\mid_\Gamma \in S^\Gamma$. The group $H^\Gamma$ is the group of all such restrictions. From now on we will consider $\Delta = \Omega - \Gamma$ and $|\Delta| = k$.
\end{remark}

The next theorem, stated and proved for the first time in \cite{WW}, characterize all the exact factorizations of the alternating group $A_n$. It is worth to mention that this theorem doesn't show in an explicit way all the  exact factorizations of the alternating group $A_n$, but it gives a procedure to follow to get all of them. From here on we consider $A^\Omega=GH$ with $G \cap H = 1$ and the sets $\Omega=\{1, \ldots, n\}$ and $\Gamma=\{1, \ldots, n-k\}$. Remind that since $\Gamma$ is $H$-stable then $\Delta$ also is and $H=H^{\Gamma} \times H^{\Delta}$.

\begin{remark}
If $q$ is a prime power then we will write $A\Gamma L(1,q)$ for the group of maps $x\rightarrow ax^{\sigma}+b$, defined over $GF(q)$, where $a, b \in GF(q)$, with $a\neq 0 $ and $\sigma \in \Aut GF(q)$. In the case $\sigma= \id$, the identity map from $\Aut GF(q)$ to itself, we have an affine transformation. The group of all such an affine maps maps will be denoted by $AGL(1,q)$. Since every element of $AGL(1,q)$ is a permutation of $GF(q)$, we'll write $ASL(1,q)$ for the group of all pair permutations in $AGL(1,q)$.
\end{remark}

\begin{theorem}{( \cite{WW} Exact Factorizations of $A^{\Omega}$)}\label{FE}
With  all the notation introduced above the following are the only ways to factorize $A^\Omega$ as an exact product $G H$: 
\begin{enumerate}
\item $H^{\Gamma}=A^{\Gamma}$, where $1\leq k \leq 5$, the group $G$ is sharply $k$-transitive over $\Omega$ and $H^{\Delta}=\{1\}$. 
\item $H^{\Gamma}=S^{\Gamma}$. In this case, all the posible factorizations are the following ones: 
\begin{center}
\begin{tabular}{ |p{1cm}||p{0.5cm}|p{0.8cm}|p{2cm}| p{2cm}|p{6.5cm}| }
 \hline
 No. & $k$ & $n$ & $G$ & $H^{\Delta}$ & Generators of $H$ \\
\hline
 $1$ & $4$ & $9$ & $PSL(2,8)$ & $S^{\Delta -\{9\}}$ & $(1,2,3,4,5)(6,7,8),(1,2)(6,7)$ \\
\hline
$2$ & $4$ & $9$ & $P\Gamma L(2,8)$ & $S^{\Delta-\{8,9\}}$ & $(1,2,3,4,5),(1,2)(6,7)$ \\
\hline 
$3$ & $4$ & $33$ & $P\Gamma L (2,32)$ & $S^{\Delta-\{33\}}$ & $(1,2,3,...,29)(30,31,32),(1,2)(30,31)$ \\
\hline 
$4$ & $3$ & $8$ & $AGL(1,8)$ & $S^{\Delta}$ & $(1,2,3,4,5)(6,7,8),(1,2)(6,7)$ \\
\hline 
$5$ &  $3$ & $8$ & $A\Gamma L(1,8)$ & $S^{\Delta-\{8\}}$ & $(1,2,3,4,5),(1,2)(6,7)$ \\
\hline 
$6$ & $3$ & $32$ & $A\Gamma L(1,32)$ & $S^{\Delta}$ & $(1,2,3,...,29)(30,31,32),(1,2)(30,31)$ \\
\hline 
$7$ & $3$ & $q+1$ & $PSL(2,q)$ & $S^{\Delta-\{q+1\}}$ & $(1,2,...,q-2),(1,2)(q-1,q)$ \\
\hline 
$8$ & $2$ & $q$ & $ASL(1,q)$ & $S^{\Delta}$ & $(1,2,...,q-2),(1,2)(q-1,q)$ \\
\hline
\end{tabular}
\end{center}
The column indexed by $k$ in the above table indicates that the group $G$ is sharply $k$-transitive and  the positive integer $q$ satisfies $q\equiv 3 \pmod{4}$.
\item If $H^{\Gamma} \neq A^{\Gamma}, S^{\Gamma}$ then $n=8$, $k=3$ and $A^{\Omega}=\mathbb{Z}_{15}A(3,2)$, where $A(3,2)$ is the affine group of order $1344$. 
\end{enumerate}

\end{theorem}
\begin{proof}
The interested reader can consult \cite[Teo. A]{WW}.
\end{proof}

The above theorem provide us with a tool to try to count the number of exact factorizations of $A_n$. In what follows we will shoy how to use it in some particular cases. Before this, we need the definition of the group of semilinear fractional trasformations that are analogous to  the group of Mobius transformations over $\mathbb{C}$ but for any field $\mathbb{F}$.

\begin{definition}\label{fractional-semi}
Given a field $\mathbb{F}$ let us consider the set $\hat{\mathbb{F}}:=\mathbb{F} \cup \{ \infty \}$, where $\infty$ is a symbol for an element that doesn't belong to  $\mathbb{F}$ (for example we can choose $\infty = \mathbb{F}$). If $\sigma \in \Aut (\mathbb{F})$ and $M$ is the matrix $\begin{bmatrix} a & b \\c & d \end{bmatrix} \in \text{GL}_2(\mathbb{F})$, then the \emph{fractional semilinear map} determined by $M$ is the map $f_M^\sigma: \hat{\mathbb{F}} \to \hat{\mathbb{F}}$ defined by the following formula
\begin{equation}
    f_M^\sigma(x):= \begin{cases}
    \dfrac{a x^\sigma + b}{c x^\sigma + d } & \text{if} \quad cx^\sigma + d \neq 0 \quad ,\\
    \infty & \text{if} \quad x = \infty \quad \text{and} \quad c = 0 \quad , \\
    a\; c^{-1} & \text{if} \quad x=\infty \quad \text{and} \quad c \neq 0.
    \end{cases}
\end{equation}
for all $x \in \hat{\mathbb{F}}$ with $cx^\sigma + d \neq 0$. In other case, if $cx^\sigma + d = 0$ then we define $f_M^\sigma(x)(x) = \infty$. 
Remind that $x^\sigma$ denotes $\sigma(x)$. If $\sigma$ is the identity morphism of $\mathbb{F}$, then $f_M^\sigma(x)$ is named  \emph{fractional linear map}, as in the case of the complex numbers.

The set of all the fractional semilinear maps is a group that is denoted by $\Gamma \text{LF}(K)$. In addition, the set of all fractional linear maps is a group denoted by $\text{LF}(K)$.
\end{definition}

To the aims of this work, the group of all fractional semilinear maps has a very important subgroup that is useful to build the Mathieu groups.

\begin{definition}
If $p$ is an odd prime number and $q=p^{2n}$ then, given an involution $\sigma \in \Aut (\text{GF}(q))$, we define the group $\textbf{\text{M}}_q$ (or $\textbf{\text{M}}(q, \sigma)$ if it is necessary to make reference to the involution) as the subgroup of $\Gamma\text{LF}(q)$ consisting of all fractional linear maps $f_M$ when $\det (M)$ is a square in $\mathbb{F}_q$, together with al the fractional semilinear maps $f_M^\sigma$ when $\det (A)$ isn't a square.
\end{definition}

A very important property of the groups $\textbf{\text{M}}_q$ is related to the degree of transitivity of them while acting as group of transformations over $\hat{\mathbb{F}}$.
 
\begin{proposition}
If $p$ is an odd prime number and $q= p^{2n}$ then $\hat{\mathbb{F}}$ is a sharply $3$-transitive $\textbf{\text{M}}_q$-set.
\end{proposition}
\begin{proof}
\cite[Teo. 9.49]{Rot}.
\end{proof}

\begin{proposition}\label{prop-A2m}
 If $m$ is an odd positive integer then 
\begin{equation}
f_2(A_{2m})=
\begin{cases}
2 & \text{If $2m=q^r+1$ with $q$ odd and $r$ even;}\\
1 & \text{If $2m=q^r+1$, but any of $q$ or $r$ don't satisfy the above conditions;}\\
0 & \text{in another case.}
\end{cases}
\end{equation}
\end{proposition}
\begin{proof}
If $A_{2m}$ denotes the alternating group in $2m$ letters, with $m$ odd, then, by the theorem \eqref{FE}, all exact factorizations $A_{2m}=G \; H$ are determined by $H^\Gamma$. Indeed, if $H^\Gamma=A^\Gamma$ then $H=H^\Gamma$ (because $H^\Delta=\{ 1 \}$) and, moreover, the action of $G$ on $\Omega$ is  sharply $k$-transitive, with $1 \leq k \leq 5$. Let's analyse independently each case for k:

{\bf Case $k = 1$:} In this case we have $|\Gamma| = 2m-1$ and  $|G| = 2m$. Then $G$ contains a permutation of order $2$ that having no fixed points. Hence, this permutation  is the composition of a number $m$ of $2$-cycles, which is a contradiction since $m$ is odd by hypothesis and this permutation must belong to the alternating group.

{\bf Case $k=2$:} In this case $\Gamma = 2m-2$ and, moreover, $G$ is doubly transitive group on a set $\Omega$ of cardinal $2m$.

According to a classical Zassenhaus' result  \cite[Teo. 20.3]{Pass}, the set of permuted points can be identified with $\text{GF}(q^r)$, the finite (Galois) field in $q^r$ elements, or $G$ have degree $5^2,7^2,11^2,23^2,29^2$ o $59^2$. Neither of these last situations is possible since $G$ has even degree. Now, in the first case, if $2m = q^r$ then $q = 2$ and  $m$ is a power of $2$, contradicting  that $m$ is odd. As result, $A_{2m}$ has no factorization in this case.

{\bf Case $k=3$:}  In this case $|\Gamma| = 2m-3$ and, moreover, $G$ is a $3$-transitive  group on $\Omega$ of order $2m(2m-1)(2m-2)$. Again, a classical Zassenhaus' result \cite[Thm 20.5]{Pass} implies that $2m-1$ must be a power $q^r$ for some prime $q$, that is, $2 m= q^r+1$. By the same result $G$ can only be $\text{LF}(q^r)$ or $\textbf{\text{M}}_{q^r}$, depending on the nature of $q$ and $r $. If $q$ is odd and $r$ is even then $G$ can be one of these two groups; otherwise, $G$ can only be $\text{LF}(q^r)$. Hence
\begin{equation*}
A^{\Omega}=
\begin{cases}
A^{\Gamma}\; \text{LF}(q^r) \quad \text{or} \quad A^{\Gamma} \; \textbf{\text{M}}_{q^r} & \text{If $q$ is odd y $r$ is even} \\
A^{\Gamma} \; \text{LF}(q^r) & \text{in another case.}
\end{cases}
\end{equation*}

{\bf Case $k = 4$:}  In this case $|\Gamma| = 2m-4$ and, furthermore, $G$ is a $4$-transitive group on $\Omega$, which is a set of cardinal $2m$. In this case, we can also apply a Jordan's theorem \cite[Thm. 21.5]{Pass} stating that $\Omega$, in this case, should be a set of cardinal $11$. Which is a contradiction.

{\bf Case $k = 5$:} In this case, the same Jordan's results \cite[Thm. 21.5]{Pass} would also imply that $\Omega$ has cardinal $12$ and therefore $m = 6$, which is a contradiction.

Now, if $H^{\Gamma}=S^{\Gamma}$ then, by theorem \cite[Thm. A]{WW} we would have $2m=q+1$ and $q\equiv 3 \mod{4}$, which is a contradiction considering $m$ is odd.
\end{proof}

The following lemma, whose proof can be found in the appendix theorem \eqref{proof-used-result}, is the main tool to determine the asymptotic behaviour of $f(A_{2^n})$, for any positive integer $n$.

\begin{lemma}\label{teo-A2n}
If $n$ is a positive integer then $f_{2}(A_{2^n})$, the number of exact factorizations of the alternating group $A_{2^n}$, is bounded by
\begin{equation}\label{cota-factor}
    2^{\frac{2}{27}n^2(n-6)} \leq f_{2}(A_{2^n}) \leq 2^{\left(\frac{2}{27} n^3 + O(n^{5/2})\right)} +\left(n2^{2n}\right)^{\left(\frac{1}{4}+o(1)\right)(2n + \log_2 n)} +1.
\end{equation}
\flushright $\square$
\end{lemma}

Since is a very difficult problem to determine the number of subgroups of a given order $m$, this difficulty is inherited by the problem to calculate $f_2(G)$ for $G$ the alternating group. The last result of the paper is an asymptotic expression for the number of exact factorizations of $A_n$ in the case $n$ is a power of $2$.

\begin{theorem}\label{teoa2n}
The number of exact factorizations of the alternating group $A_{2^n}$ satisfies
\begin{equation}
    f_2(A_{2^n})=2^{\frac{2}{27}n^3 + O(n^{5/2})}.
\end{equation}
\end{theorem}
\begin{proof}
The proof boils down to adjust, in a convenient way, the inequality \eqref{cota-factor} in the previous theorem \eqref{teo-A2n}. In fact, since
$$f_2(A_{2^n}) \leq 2^{\frac{2}{27}n^3 + O(n^{5/2})} + (n 2^{2n})^{(\frac{1}{4} + o(1))(2n+\log_2n)}+1 ,$$
then, if $A \in \mathbb{R}$ is a positive constant such that $|O(n^{\frac{5}{2}})|\leq An^{\frac{5}{2}}$, we can write 
$$      \log_2f_2(A_{2^n}) \leq \frac{2}{27}n^3+An^{\frac{5}{2}} +\log_2\left(1+ \frac{(n2^{2n})^{{(\frac{1}{4} + o(1))(2n+\log_2n)}} + 1}{2^{\frac{2}{27}n^3+An^{5/2}}} \right).$$             
\noindent Hence
\begin{align*}
  \frac{1}{n^{\frac{5}{2}}} \left(\log_2f_2(A_{2^n})-\frac{2}{27}n^3\right) \leq A+\frac{1}{n^{\frac{5}{2}}}\log_2\left(1+\frac{(n2^{2n})^{{(\frac{1}{4} + o(1))(2n+\log_2n)}}+1}{2^{\frac{2}{27}n^3+An^{\frac{5}{2}}}} \right). \\ 
\end{align*}
This expression enable us to state that $\frac{1}{n^{5/2}} \left(\log_2f_2(A_{2^n})-\frac{2}{27}n^3\right)$ is bounded from above if the argument of the logarithm function on the right is. Which, in turn, is bounded above if
$$\frac{(n2^{2n})^{{(\frac{1}{4} + o(1))(2n+\log_2n)}}}{2^{\frac{2}{27}n^3+An^{\frac{5}{2}}}} $$ 
also is. 

With $g(n)=o(1)$, there is $N_1 \in \mathbb{N}$ such that for all $n>N$ we have
 $|g(n)|\leq \frac{3}{4}$. Then, for all $n \geq N_1$
 \begin{align*}
   \frac{(n2^{2n})^{{(\frac{1}{4} + o(1))(2n+\log_2n)}}+1}{2^{\frac{2}{27}n^3+An^{\frac{5}{2}}}} &\leq \frac{(n2^{2n})^{{(2n+\log_2n)}}+1}{2^{\frac{2}{27}n^3+An^{\frac{5}{2}}}}
  \\
  &=\frac{n^{{2n+\log_2n}}2^{{4n^2+2n\log_2n}}+1}{2^{\frac{2}{27}n^3+An^{\frac{5}{2}}}}, \quad \text{but $\log_2 n < n$ then} \\
  &\leq \frac{n^{{3n}}2^{{6n^2}}+1}{2^{\frac{2}{27}n^3+An^{\frac{5}{2}}}}\\
  &\leq \frac{(2^n)^{{3n}}2^{{6n^2}}}{2^{\frac{2}{27}n^3+An^{\frac{5}{2}}}}+\frac{1}{2^{\frac{2}{27}n^3+An^{\frac{5}{2}}}} \\
  &=\frac{2^{{9n^2}}}{2^{\frac{2}{27}n^3+An^{\frac{5}{2}}}}+\frac{1}{2^{\frac{2}{27}n^3+An^{\frac{5}{2}}}} .\\
   \end{align*}

Last two terms tends to $0$ when $n \to \infty$, then there is $N_2\in \mathbb{N}$  such that for all $n>N_2$ we have 
 $$\frac{2^{{9n^2}}}{2^{\frac{2}{27}n^3+An^{\frac{5}{2}}}}+\frac{1}{2^{\frac{2}{27}n^3+An^{\frac{5}{2}}}} < 1.$$
With $M=Max\{N_1,N_2\}$, then for all $n>M$ the following holds 
 $$\frac{1}{n^{\frac{5}{2}}} \left(\log_2f_2(A_{2^n})-\frac{2}{27}n^3\right) < 1,$$

On the other side, by the left part of inequality \eqref{cota-factor} in lemma \eqref{teo-A2n}, we have $2^{\frac{2}{27}n^2(n-6)} \leq f_{2}(A_{2^n}).$
By taking logarithms we arrive at  $-\frac{12}{27}n^2 \leq \log_2 (f_{2}(A_{2^n}))-\frac{2}{27}n^3$ and
$$-\dfrac{1}{n^{1/2}}<\frac{1}{n^{\frac{5}{2}}}\left(\log_2f_2(A_{2^n})-\frac{2}{27}n^3\right).$$
Hence $\frac{1}{n^{\frac{5}{2}}}\left(\log_2f_2(A_{2^n})-\frac{2}{27}n^3\right)$ is also bounded from below.  From all above we get the desired result
 \begin{equation*}
    f_2(A_{2^n})=2^{\frac{2}{27}n^3 + O(m^{5/2})}.
\end{equation*}
\end{proof}
 
\end{section}

\section{Some questions about the function $f_2$}

In table \eqref{table-total-number} we summarize some results of the work and the total number of factorization of groups up to order $31$.

\begin{table}[h]
\begin{tabular}{|c||c|c| }
 \hline
Order $n$ & $gnu(n)$ & $\sum_{|G|=n}f_2(n)$ \\
 \hline
 $p$ prime & $1$ & $0$ \\
\hline 
$4$ & $2$ & $1$ \\
\hline 
$6$ & $2$ & $2$ \\
\hline
$8$ & $5$ & $4$ \\
\hline 
$9$ & $2$ & $1$ \\
\hline
$10$ & $2$ & $2$ \\
\hline
$12$ & $5$ & $8$ \\
\hline
$14$ & $2$ & $2$ \\
\hline 
$15$ & $1$ & $1$ \\
\hline 
$16$ & $14$ & $19$ \\
\hline 
$18$ & $5$ & $7$ \\
\hline 
$20$ & $5$ & $8$ \\
\hline
$21$ & $2$ & $2$ \\
\hline
$22$ & $2$ & $2$ \\
\hline
$24$ & $6$ & $17$ \\
\hline
$25$ & $2$ & $2$ \\
\hline 
$26$ & $2$ & $2$ \\
\hline
$27$ & $5$ & $4$ \\
\hline
$28$ & $4$ & $7$ \\
\hline 
$30$ & $4$ & $8$ \\
\hline 
\end{tabular}
\caption{Total number of exact factorizations, by order, for low order groups}
\label{table-total-number}
\end{table}

Although we know that the above list is not comprehensive enough, based on the previous discussion we can ask for the nature of the total number of factorizations of groups once the order is fixed. We have particular interest in the following questions:
\begin{itemize}
    \item If $n$ is a positive integer, is $\sum_{|G|=n}f_2(n)$ always a prime number or a power or two?
    \item If not, what is the least value of $n$ giving a counterexample? 
\end{itemize}

\begin{section}{Appendix}

In this section we proof lemma \eqref{teo-A2n} that is the main result to study the asymptotic behaviour of $f_2(A_{2^n})$.

\begin{lemma}\label{proof-used-result}
If $n$ is a positive integer then $f_{2}(A_{2^n})$, the number of exact factorizations of the alternating group $A_{2^n}$, is bounded by
\begin{equation}
    2^{\frac{2}{27}n^2(n-6)} \leq f_{2}(A_{2^n}) \leq 2^{\left(\frac{2}{27} n^3 + O(m^{5/2})\right)} +(n2^{2n})^{\left(\frac{1}{4}+o(1)\right)(2n + \log_2 n)} +1.
\end{equation}
\end{lemma}
\begin{proof}
First, we will prove that any noncyclic group of order $2^n$ is an exact factor of $A_{2^n}$. Indeed, let $H$ be a group of order $2^n$ and consider $H$ acting on itself by left multiplication. This action gives rise to a group morphism $\Phi: H \to \mathbb{S}_{2^n}$, defined, as usual, on each $h \in H$ as the map $\Phi_h: H \to H$ , given by $\Phi_h(x)=hx$, for all $x \in H$. Since $\Phi_h$ has no fixed points (except for $h=e$, the identity of the group), we can say that the representation of $\Phi_h$ is a product of disjoint cycles neither of them of length $1$ and therefore contains all natural numbers between $1$ and $2^n$.

If $h \in H$ then the order of $h$ is a power of $2$, that is $|h|=2^t$, for some $t\leq n$. Then the number $2^t$ is the smallest positive integer $\ell$ such that $\Phi_h^\ell(x)=x$, for any $x \in H$. This implies that in the representation of $\Phi_h$, as a product of disjoint cycles, we have $\Phi_h=\sigma_1 \sigma_2...\sigma_k$, where each $\sigma_i$ is a cycle of the same length $2^t$. This implies that $2^tk=2^n $, that is, $k=2^{n-t}$ and therefore $k$ is even except for $n=t$. Hence, every $\Phi_h$ is an even permutation, except for the cyclic group $H= \mathbb{Z}_{2^n}$. From all above, if $H$ is a noncyclic group of order $2^n$ then it is isomorphic to a subgroup of the alternating group $A_{2^n}$.

With the notation of remark \ref{note}, if $A_{2^n}=G \; H $ is an exact factorization then we need to analyse several cases, depending on which group is $H^\Gamma$.
If $H^\Gamma = A^\Gamma$ then, in this case, the number $k$ assume only a value between 1 and 5 and, furthermore, $H = A^\Gamma$. Depending on the values of $k$, we have the  following situations:

{\bf Case $k = 1$:} $G$ is a transitive group on $\Omega$, of order $2^n$, which acts without fixed points. According to what was shown initially, given any group of order $2^n$ we would have an exact factorization $A_{2^n}=G \; A_{2^n-1}$.

{\bf Case $k =2$:} We have that $G$ is a sharply $2$-transitive group and therefore, as a consequence of Zassenhaus' theorem (see \cite[Thm. 20.3]{Pass}), in this case, $\Omega$ can be identified with $\text{GF}(q)$, where $q= p^n$ for some prime $p$ and further $G$ is a subgroup of $A \Gamma L(1, q)$ . It is clear that in this case, $q=2^n$.

{\bf Case $k =3$:} $G$ is a sharply $3$-transitive group. Then, by the Zassenhaus' theorem \cite[20.5]{Pass}, we have that $2^n-1=p^m$, for some prime $p$ and a positive integer $m$. Furthermore, $\Omega$ can be identified with the set $\hat{K}$ \eqref{fractional-semi}, where $K=GF(q)$ and, depending on the nature of $q$ and $m$, the group $G $ has the following options:
\begin{itemize}
    \item $G$ is isomorphic to $\text{LF}(p^m)$ or $M(p^m)$, if $p$ is odd and $n$ is even;
    \item $G$ is isomorphic to $LF(p^m)$, if $p=2$ or $m$ is odd.
\end{itemize}

{\bf Case $k=4$ or $k=5$,} then the cardinal of $\Omega$ must be $11$ or $12$, respectively. Which is impossible since $|\Omega|=2^n$.

Now, if $H^\Gamma = S^\Gamma$ then, according to \eqref{FE}, there would only be two factorizations of $A_{2^n}$ for $n=3$ and one factorization for $n=5$. The same theorem indicates that there is an additional factorization for the case $n=3$, where $H^\Gamma \neq A^\Gamma, S^\Gamma$.

If we denote by $\text{gnu}(\ell)$ the number of non-isomorphic groups of order $\ell$ and by $\sigma_{k}(\ell, r)$ the number of sharply $k$-transitive groups of order $\ell$ and degree $r$, we can state that
\begin{align}\label{bound1}
    f_2(A_{2^n}) &= (\text{gnu}(2^n)-1) + \sum_{k=2}^5 \sigma_{k}\left(2^n(2^n -1)\cdots (2^n-k+1), 2^n \right) .\\
\end{align}

C. Sims, M. Newman and C. Seeley \cite[Theo. 5.8]{BNV} proved that $\text{gnu}(p^n) = p^{\frac{2}{27}m^3 + O(m^{5/2})}$. Also, Borovik, Pyber, and Shalev \cite[Cor 1.6]{BPS} proved that given any group $F$, the number of subgroups of $F$ is bounded by $|F| (\frac{1}{4} + o(1))\log_2|F|$. Hence, carrying this bounds to \eqref{bound1}, we get, for $n\geq 4$,
\begin{equation*}
f_2(A_{2^n}) \leq (2^{\frac{2}{27}n^3 + O(n^{5/2})} -1)+ |A \Gamma L(1, 2^n)|^{(\frac{1}{4} + o(1))\log_2|A \Gamma L(1.2^n)|} + 2 .
\end{equation*}
But $|A \Gamma L(1,2^n)| = n 2^n(2^n-1)$, therefore

\begin{equation*}
f_2(A_{2^n}) \leq 2^{\frac{2}{27}n^3 + O(n^{5/2})} + (n 2^n(2^n-1) )^{(\frac{1}{4} + o(1))\log_2(n 2^n(2^n-1))} + 1 ,
\end{equation*}
and the desired upper bound follows.

Now, G. Higman \cite[Theo. 4.5]{BNV} proved that $\text{gnu}(p^n) \geq p^{\frac{2}{27}n^2(n - 6)}$. This and what has been demonstrated in \textit{case I}, implies
$f_2(A_{2^n}) \geq 2^{\frac{2}{27}n^2(n - 6)}$, and the proof ends.
\end{proof}

\begin{obs}
To close this section we would like to mention a fact that went unnoticed in case III of the proof of of lemma \ref{teo-A2n}. In this case, sharply $3$-transitive groups contribute with, at most, two additional groups to the counting. But, what are the cases where this two additional cases does appear? For an exact $3$-transitive factor $G$ to exist, it is necessary that $2^n-1=p^m$, for some odd prime $p$, that is $x^n-y^m =1$ has non negative non trivial integer solutions; which, according to \emph{Catalan Conjecture} (now Mihailescu's theorem \cite{Mi}), it holds only if $m=1$. That is, $A_{2^n}$ only has the sharply $3$-transitive factor $G$ if $p=2^n - 1$ is a Mersenne's prime. 
\end{obs}

\end{section}

\begin{section}{Acknowledgments}
This work was supported by Pontificia Universidad Javeriana at Bogot\'a, Colombia, under the research project with ID 9925.
\end{section}

\bibliographystyle{plain}

\begin{thebibliography}{10}

\bibitem{AN}
Nicol\'{a}s Andruskiewitsch and Sonia Natale.
\newblock The structure of double groupoids.
\newblock {\em J. Pure Appl. Algebra}, 213(6):1031--1045, 2009.

\bibitem{AOT}
Nicolas Andruskiewitsch, Jesus Ochoa~Arango, and Alejandro Tiraboschi.
\newblock On slim double {L}ie groupoids.
\newblock {\em Pacific J. Math.}, 256(1):1--17, 2012.

\bibitem{BV}
Geo.~D. Birkhoff and H.~S. Vandiver.
\newblock On the integral divisors of {$a^n-b^n$}.
\newblock {\em Ann. of Math. (2)}, 5(4):173--180, 1904.

\bibitem{BNV}
Simon~R. Blackburn, Peter~M. Neumann, and Geetha Venkataraman.
\newblock {\em Enumeration of finite groups}, volume 173 of {\em Cambridge
  Tracts in Mathematics}.
\newblock Cambridge University Press, Cambridge, 2007.

\bibitem{BPS}
Alexandre~V. Borovik, Laszlo Pyber, and Aner Shalev.
\newblock Maximal subgroups in finite and profinite groups.
\newblock {\em Trans. Amer. Math. Soc.}, 348(9):3745--3761, 1996.

\bibitem{Con}
Keith Conrad.
\newblock Dihedral groups ii.
\newblock {https://kconrad.math.uconn.edu/grouptheory/dihedral2.pdf}.
\newblock Accessed: 2021-11-24.

\bibitem{Gus}
W.~H. Gustafson.
\newblock What is the probability that two group elements commute?
\newblock {\em Amer. Math. Monthly}, 80:1031--1034, 1973.

\bibitem{Hup}
B.~Huppert.
\newblock {\em Endliche {G}ruppen. {I}}.
\newblock Die Grundlehren der mathematischen Wissenschaften, Band 134.
  Springer-Verlag, Berlin-New York, 1967.

\bibitem{Isa}
I.~Martin Isaacs.
\newblock {\em Character theory of finite groups}.
\newblock AMS Chelsea Publishing, Providence, RI, 2006.
\newblock Corrected reprint of the 1976 original [Academic Press, New York;
  MR0460423].

\bibitem{Ito}
Noboru Ito.
\newblock On the factorizations of the linear fractional group {$LF(2,p\sp
  n)$}.
\newblock {\em Acta Sci. Math. (Szeged)}, 15:79--84, 1953.

\bibitem{Kegel}
Otto~H. Kegel.
\newblock Produkte nilpotenter {G}ruppen.
\newblock {\em Arch. Math. (Basel)}, 12:90--93, 1961.

\bibitem{Les1}
Paul Lescot.
\newblock Isoclinism classes and commutativity degrees of finite groups.
\newblock {\em J. Algebra}, 177(3):847--869, 1995.

\bibitem{Les2}
Paul Lescot.
\newblock Central extensions and commutativity degree.
\newblock {\em Comm. Algebra}, 29(10):4451--4460, 2001.

\bibitem{LPS}
Martin~W. Liebeck, Cheryl~E. Praeger, and Jan Saxl.
\newblock The maximal factorizations of the finite simple groups and their
  automorphism groups.
\newblock {\em Mem. Amer. Math. Soc.}, 86(432):iv+151, 1990.

\bibitem{Mack92}
Kirill C.~H. Mackenzie.
\newblock Double {L}ie algebroids and second-order geometry. {I}.
\newblock {\em Adv. Math.}, 94(2):180--239, 1992.

\bibitem{Ma2}
Akira Masuoka.
\newblock Calculations of some groups of {H}opf algebra extensions.
\newblock {\em J. Algebra}, 191(2):568--588, 1997.

\bibitem{Ma}
Akira Masuoka.
\newblock Hopf algebra extensions and cohomology.
\newblock In {\em New directions in {H}opf algebras}, volume~43 of {\em Math.
  Sci. Res. Inst. Publ.}, pages 167--209. Cambridge Univ. Press, Cambridge,
  2002.

\bibitem{Mi}
Preda Mih\u{a}ilescu.
\newblock Primary cyclotomic units and a proof of {C}atalan's conjecture.
\newblock {\em J. Reine Angew. Math.}, 572:167--195, 2004.

\bibitem{Pass}
Donald~S. Passman.
\newblock {\em Permutation groups}.
\newblock Dover Publications, Inc., Mineola, NY, 2012.
\newblock Revised reprint of the 1968 original.

\bibitem{Ro}
J.~Barkley Rosser and Lowell Schoenfeld.
\newblock Approximate formulas for some functions of prime numbers.
\newblock {\em Illinois J. Math.}, 6:64--94, 1962.

\bibitem{Rot}
Joseph~J. Rotman.
\newblock {\em An introduction to the theory of groups}, volume 148 of {\em
  Graduate Texts in Mathematics}.
\newblock Springer-Verlag, New York, fourth edition, 1995.

\bibitem{SaFa}
F.~Saeedi and M.~Farrokhi D.~G.
\newblock Factorization numbers of some finite groups.
\newblock {\em Glasg. Math. J.}, 54(2):345--354, 2012.

\bibitem{SuI}
Michio Suzuki.
\newblock {\em Group theory. {I}}, volume 247 of {\em Grundlehren der
  Mathematischen Wissenschaften}.
\newblock Springer-Verlag, Berlin-New York, 1982.
\newblock Translated from the Japanese by the author.

\bibitem{Su}
Michio Suzuki.
\newblock {\em Group theory. {II}}, volume 248 of {\em Grundlehren der
  mathematischen Wissenschaften [Fundamental Principles of Mathematical
  Sciences]}.
\newblock Springer-Verlag, New York, 1986.
\newblock Translated from the Japanese.

\bibitem{Tu}
Marius T\u{a}rn\u{a}uceanu.
\newblock Subgroup commutativity degrees of finite groups.
\newblock {\em J. Algebra}, 321(9):2508--2520, 2009.

\bibitem{WW}
James Wiegold and Alan~G. Williamson.
\newblock The factorisation of the alternating and symmetric groups.
\newblock {\em Math. Z.}, 175(2):171--179, 1980.

\bibitem{Zsig}
K.~Zsigmondy.
\newblock Zur {T}heorie der {P}otenzreste.
\newblock {\em Monatsh. Math. Phys.}, 3(1):265--284, 1892.

\end{thebibliography}

\end{document}